\documentclass{article}
\usepackage{amsmath,amssymb,amsfonts,amsthm,mathrsfs}
\usepackage{nameref} 
\usepackage{hyperref}
\usepackage{xcolor}
\usepackage{graphicx} 
\usepackage{mathtools} 
\usepackage{cleveref} 
\usepackage[margin=20pt]{caption} 

\usepackage{thmtools, thm-restate} 

\usepackage[style=alphabetic,sorting=nyt,maxbibnames=7,maxcitenames=5, backend=biber, style=alphabetic]{biblatex}
\newtheorem{thm}{Theorem}[section]

\newtheorem{prop}[thm]{Proposition}

\theoremstyle{definition}

\newtheorem{rem}[thm]{Remark}
\AtEndEnvironment{rem}{\null\hfill\ensuremath{\triangle}}

\newcommand{\NN}{\mathbb N}
\newcommand{\RR}{\mathbb R}

\newcommand{\Mcg}{\operatorname{MCG}} 

\newcommand\extrafootertext[1]{%
    \bgroup
    \renewcommand\thefootnote{\fnsymbol{footnote}}%
    \renewcommand\thempfootnote{\fnsymbol{mpfootnote}}%
    \footnotetext[0]{#1}%
    \egroup
}

\title{All knots are trivial: a ``proof'' by sleight of hand}

\author{Raphael Appenzeller$^1$ \and José Pedro Quintanilha$^1$}
\date{$^1$Universität Heidelberg\\ \today}

\begin{document}

\maketitle

\extrafootertext{2020 \textit{Mathematics Subject Classification}: 57K10, 00A08.}
\extrafootertext{\textit{Keywords}: knots, potholder diagrams, magic trick}

\begin{abstract}
    We take a close look at a classical magic trick performed with a string, where a trivial knot is seemingly isotoped into a trefoil, and generalize it to a family of magic tricks for transforming the unknot into other knots. We encode such a trick by depicting the target knot as a special type of knot diagram, which we call a ``knotholder diagram''. By providing an explicit construction of knotholder diagrams for all knots, we obtain variants of the trick for producing every knot. 
\end{abstract}

\section{Introduction}

On March 26, 1996, the Canadian science documentary series \emph{The Nature of Things with David Suzuki} dedicated an episode to the legacy of Martin Gardner and his endless trove of mathematical tricks, puzzles and games. One of the guests featured in this documentary was John Conway, who, at a certain point, executes a magic trick with a string, where he seems to show that the trefoil knot is trivial.
An archived version of this documentary is preserved by The Internet Archive \cite{NoT96}, with Conway's charming performance appearing at the~2:53 mark.
We recommend the reader watch this video before proceeding, but in short, Conway begins by holding the ends of an unknotted piece of string, then executes a sequence of hand motions without letting go of the ends of the string, and finally the string is revealed to be knotted into a trefoil knot. But such a maneuver is impossible, since it is well-known that a trivial knot (namely, the one formed by the string and the performer's body) is not isotopic to a trefoil.



The knot theorist bearing witness will no doubt feel the burning urge to look for an explanation, and soon enough a desperate Internet search, or a peek into Kauffman's book \cite[98]{Kau87}, reveals that the secret is, of course, a sleight of hand. Towards the end of the hand motions, the performer's hand sneakily releases the string whilst grabbing it at a nearby location, thus producing a trefoil.

This article arose from the question ``What other knots can one produce from the unknot?'', and we shall answer it with an emphatic ``All of them!''. To even make sense of such a statement, we first suggest a general 3-step scheme for what we call ``knotholder tricks'', which produce knots from an unknotted string via a sleight of hand similar to the classical trick (\Cref{sec:magic}). The precise instructions for executing a knotholder trick can be neatly presented by a special type of knot diagram for the target knot, which we introduce in \Cref{sec:diagrams} under the name ``knotholder diagram''. Knotholder diagrams for the first seven knots on Rolfson's knot table \cite{Rol76} are given in  \Cref{fig:intro_list_of_small_knots}, and two examples for composite knots are given in \Cref{fig:conn_sum}. Videos of the authors performing the corresponding tricks are available online.\footnote{\url{https://tube.mathe.social/w/p/2M22RF21tkUwC8Hm4Cr8XM}}

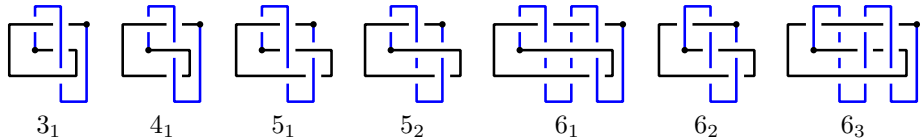
\begin{figure}[h!]
    \centering
    \def\svgwidth{1.0\linewidth}
\begingroup%
  \makeatletter%
  \providecommand\color[2][]{%
    \errmessage{(Inkscape) Color is used for the text in Inkscape, but the package 'color.sty' is not loaded}%
    \renewcommand\color[2][]{}%
  }%
  \providecommand\transparent[1]{%
    \errmessage{(Inkscape) Transparency is used (non-zero) for the text in Inkscape, but the package 'transparent.sty' is not loaded}%
    \renewcommand\transparent[1]{}%
  }%
  \providecommand\rotatebox[2]{#2}%
  \newcommand*\fsize{\dimexpr\f@size pt\relax}%
  \newcommand*\lineheight[1]{\fontsize{\fsize}{#1\fsize}\selectfont}%
  \ifx\svgwidth\undefined%
    \setlength{\unitlength}{1002.75594101bp}%
    \ifx\svgscale\undefined%
      \relax%
    \else%
      \setlength{\unitlength}{\unitlength * \real{\svgscale}}%
    \fi%
  \else%
    \setlength{\unitlength}{\svgwidth}%
  \fi%
  \global\let\svgwidth\undefined%
  \global\let\svgscale\undefined%
  \makeatother%
  \begin{picture}(1,0.14136449)%
    \lineheight{1}%
    \setlength\tabcolsep{0pt}%
    \put(0,0){\includegraphics[width=\unitlength,page=1]{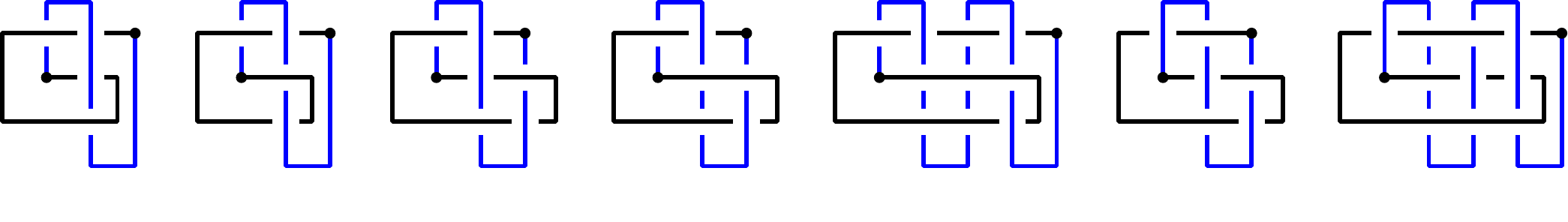}}%
    \put(0.04491434,0.00128552){\color[rgb]{0,0,0}\makebox(0,0)[t]{\lineheight{0.69999999}\smash{\begin{tabular}[t]{c}$3_1$\end{tabular}}}}%
    \put(0.16929596,0.00128552){\color[rgb]{0,0,0}\makebox(0,0)[t]{\lineheight{0.69999999}\smash{\begin{tabular}[t]{c}$4_1$\end{tabular}}}}%
    \put(0.61346556,0.00128552){\color[rgb]{0,0,0}\makebox(0,0)[t]{\lineheight{0.69999999}\smash{\begin{tabular}[t]{c}$6_1$\end{tabular}}}}%
    \put(0.76611573,0.00128552){\color[rgb]{0,0,0}\makebox(0,0)[t]{\lineheight{0.69999999}\smash{\begin{tabular}[t]{c}$6_2$\end{tabular}}}}%
    \put(0.92689328,0.00128552){\color[rgb]{0,0,0}\makebox(0,0)[t]{\lineheight{0.69999999}\smash{\begin{tabular}[t]{c}$6_3$\end{tabular}}}}%
    \put(0.3025115,0.00128552){\color[rgb]{0,0,0}\makebox(0,0)[t]{\lineheight{0.69999999}\smash{\begin{tabular}[t]{c}$5_1$\end{tabular}}}}%
    \put(0.44385426,0.00128552){\color[rgb]{0,0,0}\makebox(0,0)[t]{\lineheight{0.69999999}\smash{\begin{tabular}[t]{c}$5_2$\end{tabular}}}}%
  \end{picture}%
\endgroup%

    \caption{Knotholder diagrams for the prime knots with at most six crossings and up to mirroring. The knotholder diagram for the trefoil~$3_1$ encodes the classical trefoil trick.}
    \label{fig:intro_list_of_small_knots}
\end{figure}

\begin{figure}[h!]
    \centering
    \def\svgwidth{0.4\linewidth}
    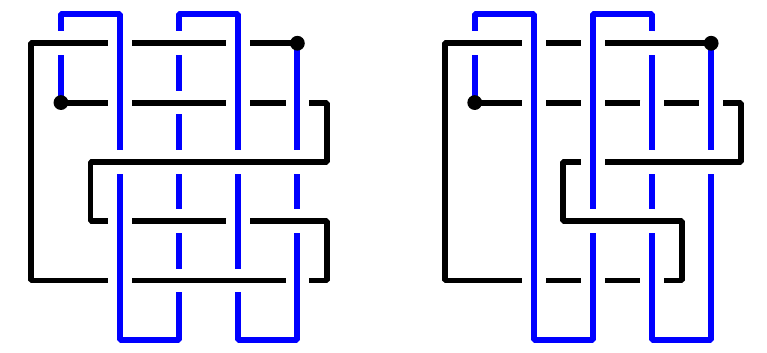
    \caption{Knotholder diagrams for the connected sum of two left-handed trefoils (left), and the connected sum of a trefoil with its mirror, also known as the square knot (right).}
    \label{fig:conn_sum}
\end{figure}

The name ``knotholder'' is a wordplay on a related class of knot diagrams, called ``potholder diagrams'', which were studied by Even-Zohar, Hass, Linial and Nowik \cite{EHL19}. Namely, they have proved that every knot admits a potholder diagram. In \Cref{sec:main_proof}, we recall their construction and explain how, given a potholder diagram, one can always produce a knotholder diagram for the same knot (\Cref{prop:potholder_to_knotholder}). This proves our main result:

\begin{restatable}[Universality of knotholder diagrams]{mainthm}{knotholderuniversality}
\label{thm:main_intro}
Every knot has a knot\-holder diagram.
\end{restatable}

We close with a brief \hyperref[sec:outlook]{outlook}, where we wonder about adapting the ideas in the article to links.

During the publication of this manuscript, an article of Behrends has appeared, which explores a generalization of the classical trick in a different direction \cite{Beh}.

\subsection*{Acknowledgements}
We are grateful to Anna Schilling for supplying the pieces of string on which we experimented throughout this project, and also for pointing out a typo in a draft of this text, to Anja Randecker, who recognised the meander diagrams among our scribbles, leading us to find the work of Even-Zohar--Hass--Linial--Nowik, to Constance Sarrazin for finding the archived version of the documentary where Conway performs the trefoil trick, and to Stefan Friedl for various discussions and for pointing out the presence of the trick in Kauffman's textbook. The second author also thanks his friend Audrey Thomas for first showing him a video of Conway's performance, which ended up kickstarting this project. Finally, we thank an anonymous referee for various corrections and comments, in particular motivating the \hyperref[sec:outlook]{outlook}.

We thank the Inkscape community for developing and maintaining the free and open-source vector graphics software used to prepare the figures in this work. The authors were partially supported by RTG 2229 ``Asymptotic Invariants and Limits of Groups and Spaces'' funded by Deutsche Forschungsgemeinschaft (DFG, German Research Foundation).



\section{The magic}\label{sec:magic}

\subsection{The trefoil trick}\label{sec:trefoil}

We begin by committing the illusionist's cardinal sin, explaining how to execute the trick that produces the left-handed trefoil from the unknot (we trust our readers to keep this secret). The instructions are given in Figure~\ref{fig:trick_explanation} and its caption;  we strongly encourage trying it out before reading on.

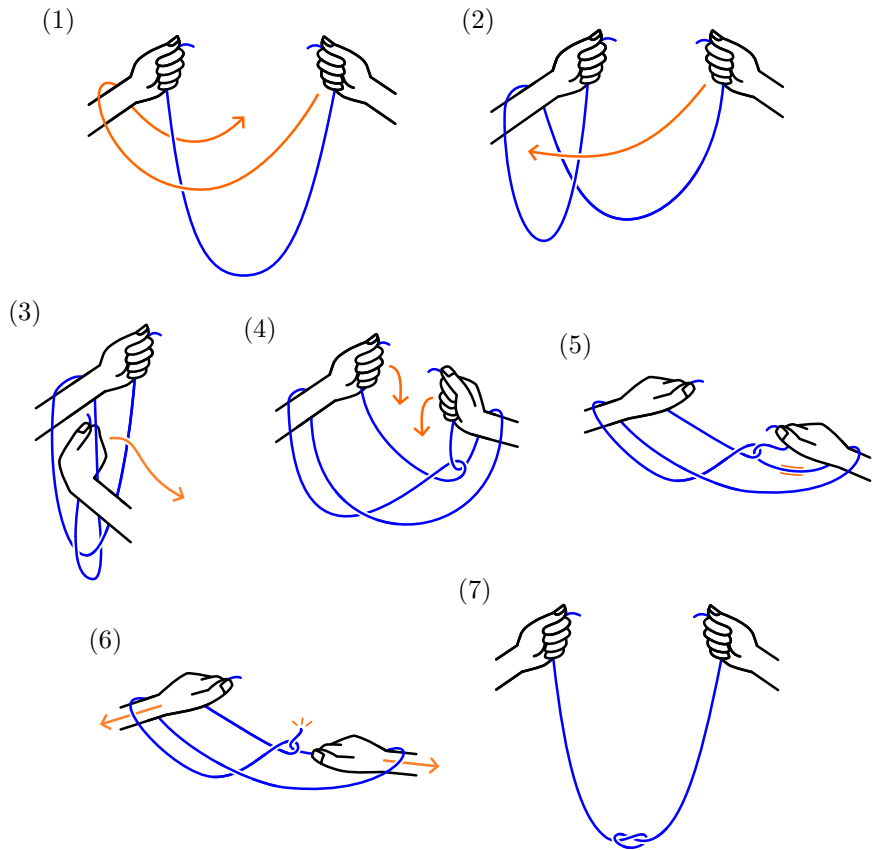
\begin{figure}[h!]
    \centering
    \def\svgwidth{\linewidth}
\begingroup%
  \makeatletter%
  \providecommand\color[2][]{%
    \errmessage{(Inkscape) Color is used for the text in Inkscape, but the package 'color.sty' is not loaded}%
    \renewcommand\color[2][]{}%
  }%
  \providecommand\transparent[1]{%
    \errmessage{(Inkscape) Transparency is used (non-zero) for the text in Inkscape, but the package 'transparent.sty' is not loaded}%
    \renewcommand\transparent[1]{}%
  }%
  \providecommand\rotatebox[2]{#2}%
  \newcommand*\fsize{\dimexpr\f@size pt\relax}%
  \newcommand*\lineheight[1]{\fontsize{\fsize}{#1\fsize}\selectfont}%
  \ifx\svgwidth\undefined%
    \setlength{\unitlength}{815.05967989bp}%
    \ifx\svgscale\undefined%
      \relax%
    \else%
      \setlength{\unitlength}{\unitlength * \real{\svgscale}}%
    \fi%
  \else%
    \setlength{\unitlength}{\svgwidth}%
  \fi%
  \global\let\svgwidth\undefined%
  \global\let\svgscale\undefined%
  \makeatother%
  \begin{picture}(1,0.97556969)%
    \lineheight{1}%
    \setlength\tabcolsep{0pt}%
    \put(0,0){\includegraphics[width=\unitlength,page=1]{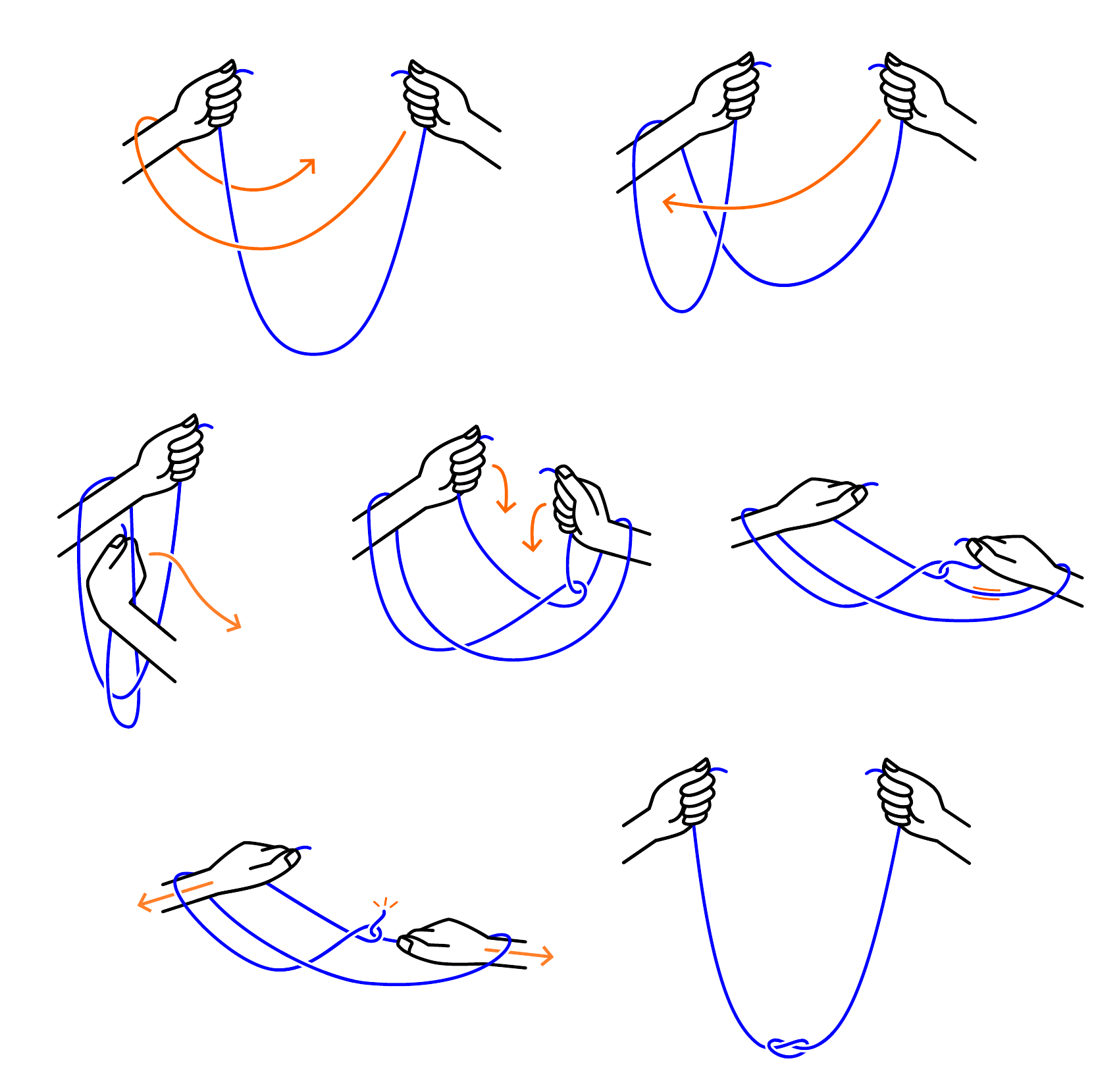}}%
    \put(0.07732902,0.9306059){\color[rgb]{0,0,0}\makebox(0,0)[t]{\lineheight{0.69999999}\smash{\begin{tabular}[t]{c}$(1)$\end{tabular}}}}%
    \put(0.0405219,0.610384){\color[rgb]{0,0,0}\makebox(0,0)[t]{\lineheight{0.69999999}\smash{\begin{tabular}[t]{c}$(3)$\end{tabular}}}}%
    \put(0.12885898,0.24967387){\color[rgb]{0,0,0}\makebox(0,0)[t]{\lineheight{0.69999999}\smash{\begin{tabular}[t]{c}$(6)$\end{tabular}}}}%
    \put(0.53373738,0.30488461){\color[rgb]{0,0,0}\makebox(0,0)[t]{\lineheight{0.69999999}\smash{\begin{tabular}[t]{c}$(7)$\end{tabular}}}}%
    \put(0.29817172,0.59198042){\color[rgb]{0,0,0}\makebox(0,0)[t]{\lineheight{0.69999999}\smash{\begin{tabular}[t]{c}$(4)$\end{tabular}}}}%
    \put(0.6441588,0.57357687){\color[rgb]{0,0,0}\makebox(0,0)[t]{\lineheight{0.69999999}\smash{\begin{tabular}[t]{c}$(5)$\end{tabular}}}}%
    \put(0.53741799,0.9306059){\color[rgb]{0,0,0}\makebox(0,0)[t]{\lineheight{0.69999999}\smash{\begin{tabular}[t]{c}$(2)$\end{tabular}}}}%
  \end{picture}%
\endgroup%

    \caption{The trefoil trick, from the performer's perspective. Starting with the unknotted string (1), the right hand moves around the left arm as illustrated, so there is now a piece of string dangling from each hand, and two from the left arm~(2). The right hand then passes through the loop formed by the string pieces dangling from the left hand and the front of the left arm, hooking around the left side of the piece of string hanging from the back of the left arm~(3), and pulling it back to produce a loop around the right wrist~(4). Then both hands curve downwards~(5), and in this motion, the right hand discretely releases the string, grabbing it at the indicated location~(6). The string is now allowed to slide off both arms, revealing the trefoil knot~(7).}
    \label{fig:trick_explanation}
\end{figure}

The transition from (5) to (6) of \Cref{fig:trick_explanation} is of course the topologically interesting step, where the knot formed by the string and the performer's body changes from an unknot to a trefoil. The situation is best understood by first observing that topologically, everything from~(3) to~(6) amounts to the right hand exchanging the segment of string it is holding, for the one dangling from the back of the left arm. Indeed, the motion of hooking around this piece of string and pulling it back serves only to conceal the exchange. In particular, it ensures that the right hand grabs the string with the thumb pointing in the correct direction, and makes it so that the excess string released is short enough to pass unnoticed. Were one to disregard these practical matters, the exchange could be executed simply as illustrated in \Cref{fig:pics_to_diagrams}, where a diagrammatic translation of the change of knots is also shown.

\begin{figure}[h]
    \centering
    \def\svgwidth{0.5\linewidth}
\begingroup%
  \makeatletter%
  \providecommand\color[2][]{%
    \errmessage{(Inkscape) Color is used for the text in Inkscape, but the package 'color.sty' is not loaded}%
    \renewcommand\color[2][]{}%
  }%
  \providecommand\transparent[1]{%
    \errmessage{(Inkscape) Transparency is used (non-zero) for the text in Inkscape, but the package 'transparent.sty' is not loaded}%
    \renewcommand\transparent[1]{}%
  }%
  \providecommand\rotatebox[2]{#2}%
  \newcommand*\fsize{\dimexpr\f@size pt\relax}%
  \newcommand*\lineheight[1]{\fontsize{\fsize}{#1\fsize}\selectfont}%
  \ifx\svgwidth\undefined%
    \setlength{\unitlength}{308.84138453bp}%
    \ifx\svgscale\undefined%
      \relax%
    \else%
      \setlength{\unitlength}{\unitlength * \real{\svgscale}}%
    \fi%
  \else%
    \setlength{\unitlength}{\svgwidth}%
  \fi%
  \global\let\svgwidth\undefined%
  \global\let\svgscale\undefined%
  \makeatother%
  \begin{picture}(1,1.12840996)%
    \lineheight{1}%
    \setlength\tabcolsep{0pt}%
    \put(0,0){\includegraphics[width=\unitlength,page=1]{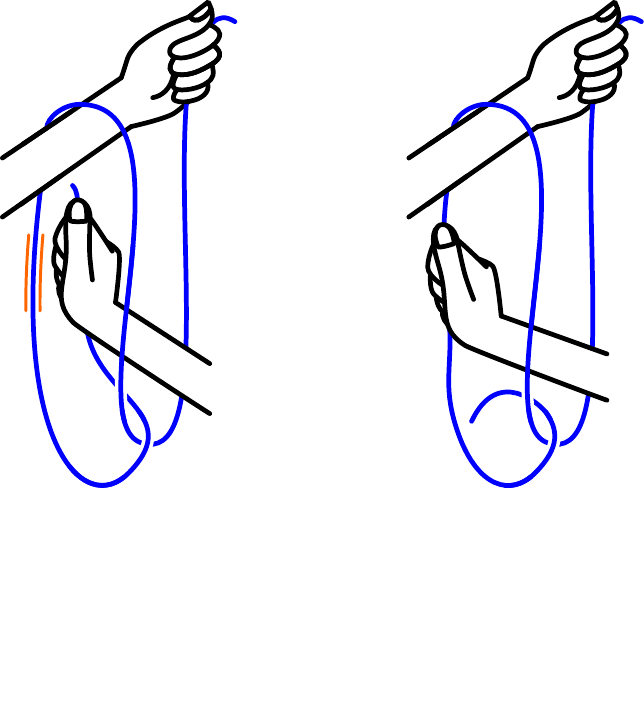}}%
    \put(0.48929959,0.7457038){\color[rgb]{0,0,0}\makebox(0,0)[t]{\lineheight{0.69999999}\smash{\begin{tabular}[t]{c}\huge $\rightsquigarrow$\end{tabular}}}}%
    \put(0,0){\includegraphics[width=\unitlength,page=2]{pics_to_diagrams.pdf}}%
  \end{picture}%
\endgroup%

    \caption{A simplified version of the change of knots. The illustration on the top left is  a rearrangement of step~(3) of Figure~\ref{fig:trick_explanation}. On the bottom, these illustrations are translated into knot diagrams -- on the left is a diagram of the unknot, and on the right, ignoring the dotted arc, we see a diagram for the trefoil. The latter is in fact an example of a knot\-holder diagram, which we will define in \Cref{sec:knotholder_diagrams}.}
    \label{fig:pics_to_diagrams}
\end{figure}

\newpage 
\phantom{.}
\newpage 

\subsection{Knotholder tricks} \label{sec:knotholder_tricks}

Let us now generalize the trefoil trick to a class of tricks that produce other knots. We shall refer to them as \textbf{knotholder tricks}, and they consist of three basic steps, illustrated in Figure~\ref{fig:instructions_5_2} (left):

    \begin{enumerate}
        \item \label{step:U_shapes}
        Starting from the unknotted string, plainly held before the audience, the right hand winds back and forth around the left arm, from right to left, to produce a sequence of dangling $\cup$-shaped pieces of string. The right-most vertical strand is being held by the left hand, and the other vertical strands might be dangling from the front or the back of the left arm.
        
        \item \label{setep:weave}
        The right hand weaves through the dangling $\cup$-shapes in an upward motion. Its trajectory should end immediately to the right of the left-most dangling strand (which is not part of a $\cup$-shape), but this strand should not be interacted with throughout the weaving motion.
        
        \item \label{step:exchange} A sleight of hand is performed, where the right hand exchanges the piece of string it is holding, with the left-most string dangling from the left arm. The details for executing this sleight of hand are explained below. 
    \end{enumerate}

 \newpage 

\begin{figure}[h]
    \centering
    \def\svgwidth{0.6\linewidth}
    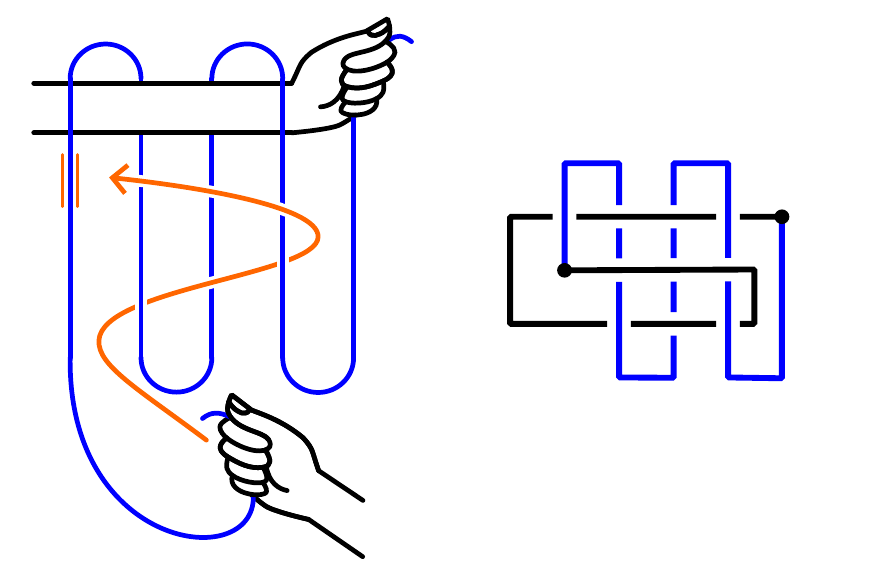
    \caption{A knotholder trick for the knot~$5_2$. Left: The string as produced at the end of Step~\ref{step:U_shapes} (with two $\cup$-shapes), and the weaving motion of the right hand in Step~\ref{setep:weave}. Step~\ref{step:exchange} has no additional input, but for the sake of clarity we indicate the region grabbed by the right hand when performing the slight of hand. Right: The corresponding knotholder diagram, where we also suggestively indicate which parts of the knot correspond to the string and to the performer's body.}
    \label{fig:instructions_5_2}
\end{figure}

    \Cref{fig:instructions_5_2} (right) exemplifies how the data for executing these steps may be succinctly specified by what we call a ``knotholder diagram'' of the target knot, whose formal definition we withhold until \Cref{sec:knotholder_diagrams}.

    We now take a moment to clarify the execution of Step~\ref{step:exchange}.  The right wrist hooks around the left of the left-most dangling piece of string, and pulls it back so it loops around the right wrist.
    Ideally, this escaping motion is performed without bringing along additional bits of string, so that at the end, only one loop of string is caught around the right wrist, as in \Cref{fig:escape} (top).

\begin{figure}[h!]
    \centering
    \def\svgwidth{0.7\linewidth}
\begingroup%
  \makeatletter%
  \providecommand\color[2][]{%
    \errmessage{(Inkscape) Color is used for the text in Inkscape, but the package 'color.sty' is not loaded}%
    \renewcommand\color[2][]{}%
  }%
  \providecommand\transparent[1]{%
    \errmessage{(Inkscape) Transparency is used (non-zero) for the text in Inkscape, but the package 'transparent.sty' is not loaded}%
    \renewcommand\transparent[1]{}%
  }%
  \providecommand\rotatebox[2]{#2}%
  \newcommand*\fsize{\dimexpr\f@size pt\relax}%
  \newcommand*\lineheight[1]{\fontsize{\fsize}{#1\fsize}\selectfont}%
  \ifx\svgwidth\undefined%
    \setlength{\unitlength}{535.16285814bp}%
    \ifx\svgscale\undefined%
      \relax%
    \else%
      \setlength{\unitlength}{\unitlength * \real{\svgscale}}%
    \fi%
  \else%
    \setlength{\unitlength}{\svgwidth}%
  \fi%
  \global\let\svgwidth\undefined%
  \global\let\svgscale\undefined%
  \makeatother%
  \begin{picture}(1,0.76357503)%
    \lineheight{1}%
    \setlength\tabcolsep{0pt}%
    \put(0,0){\includegraphics[width=\unitlength,page=1]{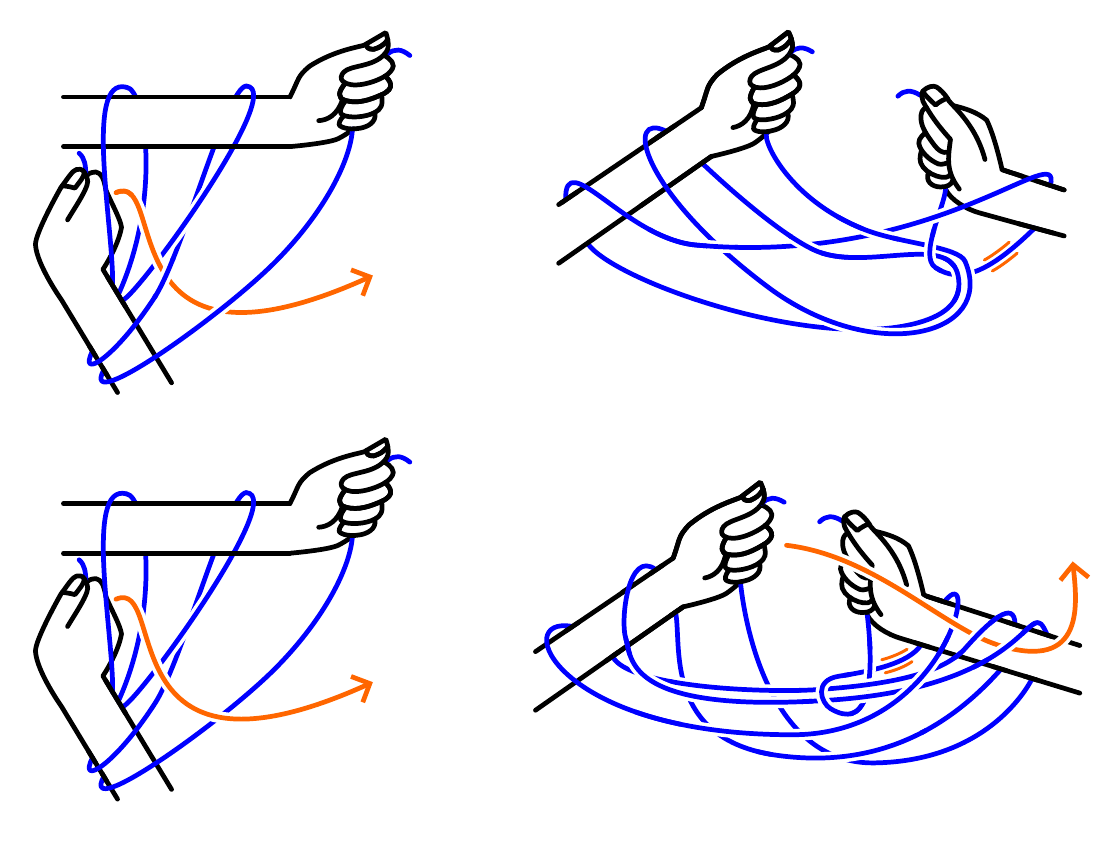}}%
    \put(0.41374875,0.56432298){\color[rgb]{0,0,0}\makebox(0,0)[t]{\lineheight{0.69999999}\smash{\begin{tabular}[t]{c}\huge $\rightsquigarrow$\end{tabular}}}}%
    \put(0.41374875,0.19994764){\color[rgb]{0,0,0}\makebox(0,0)[t]{\lineheight{0.69999999}\smash{\begin{tabular}[t]{c}\huge $\rightsquigarrow$\end{tabular}}}}%
  \end{picture}%
\endgroup%

    \caption{The sleight of hand in Step~\ref{step:exchange}, following up on the example from \Cref{fig:instructions_5_2}. Top: Ideally, after the right wrist hooks around the left-most strand dangling from the left arm, only this strand is dragged along, and the right hand can then execute the exchange. Bottom: If undesired bits of string are caught in the escaping motion, the extraneous loops around the right arm may be scooped away by the left arm.}
    \label{fig:escape}
\end{figure}
    
It might however be unavoidable that some additional undesired pieces of string loop around the right wrist, and these other loops might interfere with the performer's ability to swap the strings as intended. In that case, the left arm may carefully scoop out the undesired bits of string from around the right arm before proceeding. This maneuver is illustrated in \Cref{fig:escape} (bottom). Note that if the hooking motion of the right arm was performed with some care, the loop closest to the hand is the one to be kept.
    
    Having only one loop of string around it, the right wrist bends downward (together with the left wrist, diverting the audience's attention), and the right hand performs the exchange as shown in \Cref{fig:trick_explanation} (5) -- (6). Letting all the string slide out of both arms, the target knot is revealed to the audience.

\begin{rem}[Loose ends]
    After executing the sleight of hand, the right hand will be holding a longer piece of loose string than in the beginning, which the audience might notice by comparison with the one on the left hand. To help mitigate this difference, the performer may begin the trick with a slightly longer tip dangling from the left hand than from the right. Even if there is a noticeable difference at this point, the audience's suspicion has not yet been aroused.
\end{rem}

\begin{rem}[Band surgery]\label{rem:band_surgery}
    A \textbf{band} for a link $L\subset \mathbb S^3$ is an embedding of a square $b\colon [-1,1]^2 \to \mathbb S^3$ that meets~$L$ precisely at two opposite sides, namely, $b^{-1}(L) = \{-1,1\} \times [-1,1]$.
    The \textbf{band surgery} of~$L$ along~$b$ is the link obtained from~$L$ by replacing $b(\{-1,1\} \times [-1,1])$ with $b([-1,1] \times \{-1,1\})$.
    
    \Cref{fig:band_surgery} illustrates that the sleight of hand in the knotholder trick may be interpreted as transforming the starting unknot by a band surgery, to produce a $2$-component link of the target knot and an unknot. The fact that this second component is indeed an unknot follows from  the right hand moving upwards during Step~\ref{setep:weave}.
\end{rem}

\begin{figure}[h!]
    \centering
    \def\svgwidth{0.54\linewidth}
\begingroup%
  \makeatletter%
  \providecommand\color[2][]{%
    \errmessage{(Inkscape) Color is used for the text in Inkscape, but the package 'color.sty' is not loaded}%
    \renewcommand\color[2][]{}%
  }%
  \providecommand\transparent[1]{%
    \errmessage{(Inkscape) Transparency is used (non-zero) for the text in Inkscape, but the package 'transparent.sty' is not loaded}%
    \renewcommand\transparent[1]{}%
  }%
  \providecommand\rotatebox[2]{#2}%
  \newcommand*\fsize{\dimexpr\f@size pt\relax}%
  \newcommand*\lineheight[1]{\fontsize{\fsize}{#1\fsize}\selectfont}%
  \ifx\svgwidth\undefined%
    \setlength{\unitlength}{411.89768393bp}%
    \ifx\svgscale\undefined%
      \relax%
    \else%
      \setlength{\unitlength}{\unitlength * \real{\svgscale}}%
    \fi%
  \else%
    \setlength{\unitlength}{\svgwidth}%
  \fi%
  \global\let\svgwidth\undefined%
  \global\let\svgscale\undefined%
  \makeatother%
  \begin{picture}(1,0.41535256)%
    \lineheight{1}%
    \setlength\tabcolsep{0pt}%
    \put(0,0){\includegraphics[width=\unitlength,page=1]{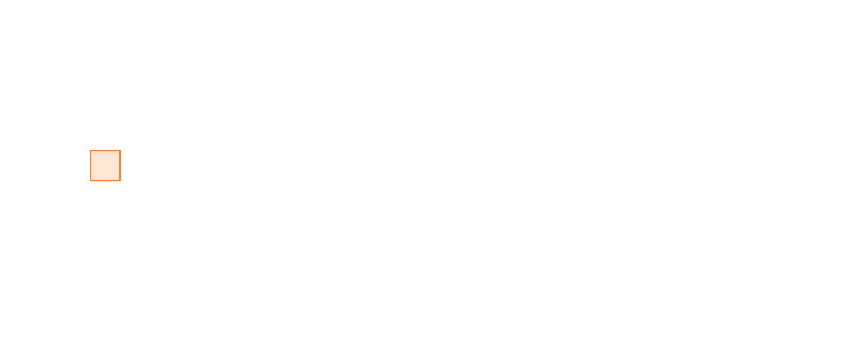}}%
    \put(0.5049648,0.19892531){\color[rgb]{0,0,0}\makebox(0,0)[t]{\lineheight{0.69999999}\smash{\begin{tabular}[t]{c}\huge $\rightsquigarrow$\end{tabular}}}}%
    \put(0,0){\includegraphics[width=\unitlength,page=2]{band_surgery.pdf}}%
  \end{picture}%
\endgroup%

    \caption{The knots before and after executing the sleight of hand for the knot\-holder trick in \Cref{fig:instructions_5_2}. The highlighted band guides a surgery from the starting unknot, to a link of the knot~$5_2$ and an un\-knot (dotted).}
    \label{fig:band_surgery}
\end{figure}

\begin{rem}[Moving up]
    Our insistence that the right hand move upwards during Step~\ref{setep:weave} is important for a few reasons. This requirement makes it so that the performer's right hand only needs to interact with the bits of string dangling from the left arm, which are easily accessible, as opposed to the portion of string below, which might get messily tangled up. Moreover, this prevents the left-over piece of string after Step~\ref{step:exchange} from getting tangled up with itself and the performer's arm.    
\end{rem}

\section{The diagrams}\label{sec:diagrams}

As already mentioned, a knot\-holder diagram for a knot~$K$ is meant to specify instructions for executing a knot\-holder trick that produces~$K$. This amounts to encoding the data of how many $\cup$-shapes should be produced in Step~\ref{step:U_shapes}, whether their vertical strands hang from the front or the back of the left arm, and how the right hand should weave amid these strands in Step~\ref{setep:weave}. This last piece of data will be displayed within the knot\-holder diagram by means of what will call a ``weaver diagram'', which represents a ``weaver tangle'' within the knot. Let us now introduce these concepts.

\subsection{Weaver tangles and weaver diagrams}\label{sec:weaver}
    Given $n\in \NN$, let $D\subset \RR^2$ be the disk that has the line segment $[0,n+1] \times \{0\}$ as a diameter. We also define the points
    \[x_1\coloneq (1,0), \ x_2\coloneq (2,0), \  \ldots, \ x_n\coloneq (n,0).\]
    A \textbf{weaver tangle} of width~$n$ consists of $n+1$~disjoint properly embedded arcs~$L_1, \ldots , L_n, \gamma$ in the cylinder~$C\coloneq D \times [0,1]$, where:
    \begin{itemize}
        \item Each~$L_i$ is the straight line segment $\{x_i\} \times [0,1]$. The union $L:=\bigcup_{i=1}^n L_i$ is called the \textbf{warp} of the weaver tangle.
        
        \item The strand~$\gamma$, which we call the \textbf{weft}\footnote{To distinguish the warp from the weft, it might be useful to remember the mnemonic that ``the weft goes weft and right, the warp goes warp and down''.}, has endpoints $(0,0,\tfrac 13)$ and $(0,0,\tfrac23)$, and admits a parameterization $[0,1] \to C = D \times [0,1]$ whose $[0,1]$-coordinate is nondecreasing.
    \end{itemize} 

    We are generally only interested in weaver tangles up to isotopy of the weft~$\gamma$ within $C \setminus L$, relative endpoints, and will say that two weaver tangles are \textbf{isotopic} if their wefts are isotopic in this sense.
    Every weaver tangle~$W$ may be isotoped so that for the projection $\pi \colon C \to [0,n+1] \times [0,1]$ to the 1st and 3rd coordinates, $\pi(\gamma)$~is an embedded arc transverse to~$\pi(L)$. Such a projection, together with over/undercrossing data at every point in~$\pi (\gamma) \cap \pi (L)$, is called a \textbf{weaver diagram} of width~$n$ for~$W$; see \Cref{fig:weaver} for an example. Every weaver tangle is determined up to isotopy by any of its diagrams.

\begin{figure}[h!]
    \centering
    \def\svgwidth{0.7\linewidth}
\begingroup%
  \makeatletter%
  \providecommand\color[2][]{%
    \errmessage{(Inkscape) Color is used for the text in Inkscape, but the package 'color.sty' is not loaded}%
    \renewcommand\color[2][]{}%
  }%
  \providecommand\transparent[1]{%
    \errmessage{(Inkscape) Transparency is used (non-zero) for the text in Inkscape, but the package 'transparent.sty' is not loaded}%
    \renewcommand\transparent[1]{}%
  }%
  \providecommand\rotatebox[2]{#2}%
  \newcommand*\fsize{\dimexpr\f@size pt\relax}%
  \newcommand*\lineheight[1]{\fontsize{\fsize}{#1\fsize}\selectfont}%
  \ifx\svgwidth\undefined%
    \setlength{\unitlength}{288.05812806bp}%
    \ifx\svgscale\undefined%
      \relax%
    \else%
      \setlength{\unitlength}{\unitlength * \real{\svgscale}}%
    \fi%
  \else%
    \setlength{\unitlength}{\svgwidth}%
  \fi%
  \global\let\svgwidth\undefined%
  \global\let\svgscale\undefined%
  \makeatother%
  \begin{picture}(1,0.48739482)%
    \lineheight{1}%
    \setlength\tabcolsep{0pt}%
    \put(0.1590534,0.00447501){\color[rgb]{0,0,1}\makebox(0,0)[t]{\lineheight{0.69999999}\smash{\begin{tabular}[t]{c}$L_1$\end{tabular}}}}%
    \put(0.26853894,0.00447501){\color[rgb]{0,0,1}\makebox(0,0)[t]{\lineheight{0.69999999}\smash{\begin{tabular}[t]{c}$L_2$\end{tabular}}}}%
    \put(0.37802509,0.00447501){\color[rgb]{0,0,1}\makebox(0,0)[t]{\lineheight{0.69999999}\smash{\begin{tabular}[t]{c}$L_3$\end{tabular}}}}%
    \put(0.0986534,0.29703448){\color[rgb]{0,0,0}\makebox(0,0)[t]{\lineheight{0.69999999}\smash{\begin{tabular}[t]{c}$\gamma$\end{tabular}}}}%
    \put(0,0){\includegraphics[width=\unitlength,page=1]{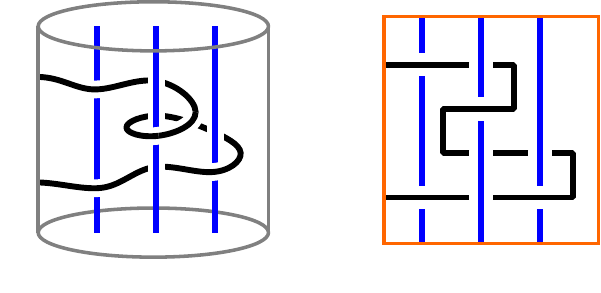}}%
  \end{picture}%
\endgroup%

    \caption{A weaver tangle of width 3 (left) and one of its weaver diagrams (right).}
    \label{fig:weaver}
\end{figure}
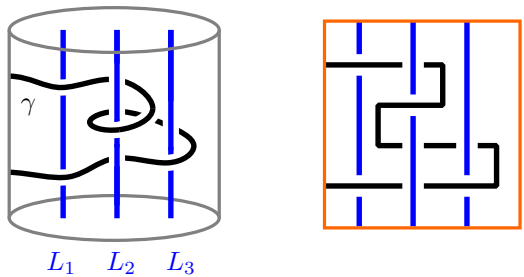

    Now let $f\colon D \to D$ be a diffeomorphism that fixes the boundary~$\partial D$ pointwise and $\{x_1, \ldots ,x_n\}$ as a set. In other words, $f$~represents an element~$[f]$ of the mapping class group $\Mcg(D, \{x_1, \ldots, x_n\})$. Given a weaver tangle $W=\gamma \cup L$ of width~$n$, we define a new weaver tangle $f\cdot W$ as the image of~$W\subset C = D \times [0,1]$ under the map $f\times \mathrm {id}$. 

    If a different map~$f'$ in the same mapping class is used, then a witnessing isotopy from~$f$ to~$f'$ shows that $f\cdot W$ is isotopic to $f'\cdot W$ through weaver tangles. Thus, since we are only interested in weaver tangles up to isotopy, we will often simply specify $[f]\cdot W$. Moreover, the well-known isomorphism $\Mcg(D, \{x_1, \ldots, x_n\}) \cong \mathcal B_n$, where $\mathcal B_n$ is the braid group on $n$~strands \cite[Section I.9.1.3]{FaMa11} allows us to conveniently describe~$[f]$ by means of a braid. An example of a braid (read from top to bottom) acting on a weaver tangle is illustrated in \Cref{fig:weaver_action}. 

\begin{figure}[h!]
    \centering
    \def\svgwidth{0.75\linewidth}
\begingroup%
  \makeatletter%
  \providecommand\color[2][]{%
    \errmessage{(Inkscape) Color is used for the text in Inkscape, but the package 'color.sty' is not loaded}%
    \renewcommand\color[2][]{}%
  }%
  \providecommand\transparent[1]{%
    \errmessage{(Inkscape) Transparency is used (non-zero) for the text in Inkscape, but the package 'transparent.sty' is not loaded}%
    \renewcommand\transparent[1]{}%
  }%
  \providecommand\rotatebox[2]{#2}%
  \newcommand*\fsize{\dimexpr\f@size pt\relax}%
  \newcommand*\lineheight[1]{\fontsize{\fsize}{#1\fsize}\selectfont}%
  \ifx\svgwidth\undefined%
    \setlength{\unitlength}{370.09743037bp}%
    \ifx\svgscale\undefined%
      \relax%
    \else%
      \setlength{\unitlength}{\unitlength * \real{\svgscale}}%
    \fi%
  \else%
    \setlength{\unitlength}{\svgwidth}%
  \fi%
  \global\let\svgwidth\undefined%
  \global\let\svgscale\undefined%
  \makeatother%
  \begin{picture}(1,0.41974974)%
    \lineheight{1}%
    \setlength\tabcolsep{0pt}%
    \put(0,0){\includegraphics[width=\unitlength,page=1]{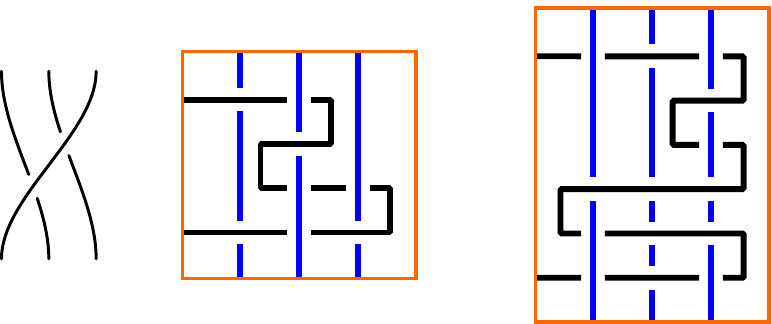}}%
    \put(0.6172285,0.19511617){\color[rgb]{0,0,0}\makebox(0,0)[t]{\lineheight{0.69999999}\smash{\begin{tabular}[t]{c}$=$\end{tabular}}}}%
    \put(0.17116076,0.19505306){\color[rgb]{0,0,0}\makebox(0,0)[t]{\lineheight{0.69999999}\smash{\begin{tabular}[t]{c}$\cdot$\end{tabular}}}}%
    \put(0,0){\includegraphics[width=\unitlength,page=2]{weaver_action.pdf}}%
  \end{picture}%
\endgroup%

    \caption{This braid acts on the weaver tangle of~\Cref{fig:weaver} by moving the third vertical segment of the warp across the front of the tangle, all the way to the first position, while the other two segments are shifted one position to the right. The weft is dragged along this motion.}
    \label{fig:weaver_action}
\end{figure}

    \subsection{Knotholder diagrams}\label{sec:knotholder_diagrams}

    Given $m\in \NN$, a \textbf{knotholder diagram} with $m$~$\cup$-shapes is a knot diagram as in \Cref{fig:knotholder_def} 
    where:
    \begin{itemize}
        \item the $W$-labeled rectangle consists of a weaver diagram of width~$2m$, and
        \item the $2m$~crossings along the long horizontal segment, whose over- and understrands are unspecified in the picture, satisfy the following requirement: for every $i\in \{1,\ldots m\}$, the crossings numbered $2i-1$ and~$2i$ have the same sign. In other words, the vertical strand in the crossing~$2i-1$ crosses over the horizontal strand if and only if the vertical strand in position~$2i$ crosses under.
    \end{itemize} 

\begin{figure}[h!]
    \centering
    \def\svgwidth{0.5\linewidth}
    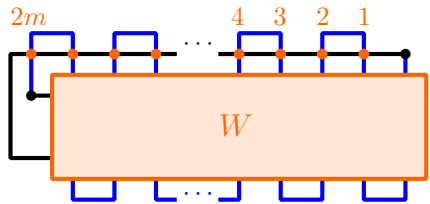
    \caption{The scheme for constructing a knothoder diagram with $m$~$\cup$-shapes. The missing data is a weaver diagram~$W$ of width~$2m$, and the signs of the consecutive pairs of crossings that are connected by $\cap$-shapes above the long horizontal segment.}
    \label{fig:knotholder_def}
\end{figure}
    
    A diagram of this form specifies a knotholder trick as illustrated in \Cref{fig:instructions_5_2}. The second condition implements the physical constraint that, at the end of Step~\ref{step:U_shapes} of the trick, the string forming the $\cup$-shapes must rest atop the performer's left arm. We have opted to number these crossings from right to left because this is the order in which they are formed during Step~\ref{step:U_shapes}.

    We introduce one last piece of notation. For each $i\in \{1, \ldots m\}$, let $\varepsilon_i\in \{+,-\}$ be the common sign of the crossings~$2i-1$ and~$2i$. These are assembled into the \textbf{sign tuple} $(\varepsilon_1, \ldots, \varepsilon_m)$ of the knotholder diagram. Physically, the fact that $\varepsilon_i=+$ (respectively $\varepsilon_i=-$) means that when the right hand moves over the left arm for the $i$-th time in Step~\ref{step:U_shapes}, this motion is towards (respectively away from) the performer's body. 
    \Cref{fig:conn_sum} shows examples with $\varepsilon = (-,-)$ (left) and $\varepsilon = (+,-)$ (right), and  in \Cref{fig:instructions_5_2} we have $\varepsilon=(-,+)$.

\begin{rem}[Mirror images]
    Switching the signs of all crossings in a knot\-holder diagram for a knot~$K$ yields a knot\-holder diagram for its mirror image~$\overline K$. This way one may easily modify a knot\-holder trick for a knot~$K$ to one for~$\overline K$ without interchanging the roles played by the two hands.
\end{rem}

\section{The proof}\label{sec:main_proof}

Equipped with the definition of knotholder diagrams, we begin tackling the proof of \Cref{thm:main_intro}. A key intermediate step involves representing knots by ``potholder diagrams'', which we now discuss.

\subsection{Potholder diagrams}

Let $n_1$ and $n_2$ be odd positive integers. A \textbf{potholder diagram} of \textbf{width}~$n_1$ and \textbf{height}~$n_2$ is a knot diagram produced from a grid of $n_1$~vertical and $n_2$~horizontal line segments by connecting each of the $2(n_1 + n_2)$ loose ends to a neighboring one and assigning over/undercrossings at each of the $n_1 n_2$ intersections. The definition is best understood by staring at a few examples, see Figure~\ref{fig:potholder}.

\begin{figure}[h!]
    \centering
    \def\svgwidth{0.6\linewidth}
    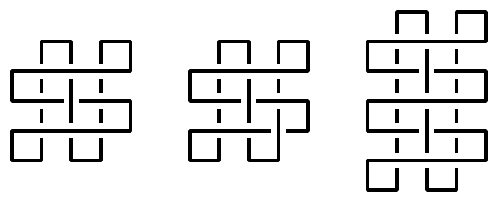
    \caption{Potholder diagrams for the right-handed trefoil, the knot~$5_2$, and the torus knot~$5_1$ (from left to right). All these potholder diagrams have width~$3$, the first two have height~$3$, and the third one has height~$5$.}
    \label{fig:potholder}
\end{figure}

Potholder diagrams were defined by Even-Zohar, Hass, Linial and Nowik \cite{EHL19} with the additional requirement that the diagram is ``square'', that is, the width and height should be the same. One can always enlarge a potholder diagram of a given knot in either direction by adding superfluous crossings, so in particular, it is possible to enlarge every potholder diagram so it becomes square. In that article, the following theorem is established \cite[Theorem~1.1]{EHL19}.

\begin{thm}[Universality of potholder diagrams]\label{thm:universality_potholder}
    Every knot has a potholder diagram.
\end{thm}

The proof proceeds by a three-step construction, which (1) modifies a given knot diagram to a ``meander diagram'' of the same knot~$K$, then (2) upgrades it to a ``standard meander diagram'', from which it finally (3) constructs a potholder diagram. For the sake of completeness, we will now give the definition of these types of diagrams and explain the three steps of the construction.

\vspace{5pt}
\underline{Step 1: Meander diagrams.}
A \textbf{meander diagram} is a knot diagram whose underlying immersed curve is the union of two embedded arcs with the same pair of endpoints, which are called \textbf{change points}. In other words, the diagram decomposes into two arcs that have no self-crossings. Potholder diagrams are examples of meander diagrams, with the top-right and bottom-left corners as change points. Knotholder diagrams too are meander diagrams, as the usual coloring in our figures shows.

To produce a meander diagram from an arbitrary diagram of~$K$, we start by decomposing its underlying curve into two immersed arcs $\alpha$ and~$\beta$ (with endpoints away from the crossings), possibly with (forbidden) self-crossings. If $\alpha$~has self-crossings, consider the first such crossing that is encountered as one walks along~$\alpha$ starting from one of the change points~$p$. To remove it, we drag the encountered $\alpha$-strand back along the initial segment of~$\alpha$, past~$p$, so it now intersects~$\beta$ instead; see \Cref{fig:diagram_to_meander} for an example. If the initial segment of~$\alpha$ had crossings with~$\beta$, this move also produces, for each such crossing, two new crossings between $\alpha$ and~$\beta$, but no forbidden crossings are created.

\begin{figure}[h]
    \centering
    \def\svgwidth{0.9\linewidth}
    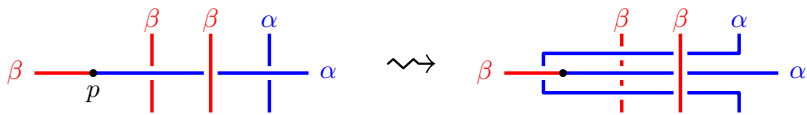
    \caption{Removing a self-intersection of~$\alpha$ in Step~1. }
    \label{fig:diagram_to_meander}
\end{figure}

Repeating this move until all self-crossings of~$\alpha$ have been eliminated, and then proceeding similarly with the self-crossings of~$\beta$, yields a meander diagram for~$K$.

\vspace{5pt}
\underline{Step 2: Standard meander diagrams.}
A meander diagram  is \textbf{standard} if there is an embedded arc~$\gamma \in \mathbb S^2$ connecting its change points, whose interior is disjoint from the diagram. Again potholder diagrams are examples, but knotholder diagrams are generally not.

Given a meander diagram~$D$ for~$K$ with a decomposition into arcs~$\alpha$, $\beta$ and change points~$p$, $q$, let $q_\alpha$~be an interior point of~$\alpha$ lying closer to~$q$ than any crossing of~$D$. The sub-arc of~$\alpha$ from~$q_\alpha$ to~$p$ will be denoted by~$\alpha_-$. Similarly, let $q_\beta \in \beta$ be an interior point close to~$q$ and define the sub-arc~$\beta_-$ from~$q_\beta$ to~$p$. Moreover, let~$\gamma \subset \mathbb S^2$ be any embedded arc having~$p$ as an endpoint, and being otherwise disjoint from~$D$. We denote its second endpoint by~$q'$. Since embedded arcs do not separate~$\mathbb S^2$, there is an embedded arc~$\alpha'$ from~$q_\alpha$ to~$q'$ whose interior is disjoint from~$\alpha_-\cup \gamma$. Similarly, there is an embedded arc~$\beta'$ from~$q_\beta$ to~$q'$ with interior disjoint from~$\beta_- \cup \gamma$.

We produce a new meander diagram~$D'$ by taking the union of the embedded arcs $\alpha_- \cup \alpha'$ and $\beta_- \cup \beta'$,
and declaring that strands from~$\alpha'$ always cross over~$\beta_-$, strands from~$\beta'$ always cross over $\alpha_-$ and~$\alpha'$, and crossings involving $\alpha_-$ and~$\beta_-$ are as in~$D$; see \Cref{fig:standard_meander} for an example. 
This choice ensures that~$D'$ still represents~$K$, because the new portion $\alpha' \cup \beta'$ describes a curve that lies unknotted above~$K$. Moreover, $D'$~has change points $p$~and~$q'$, so $\gamma$~witnesses that $D'$~is a standard meander diagram.

\begin{figure}[h]
    \centering
    \def\svgwidth{0.9\linewidth}
    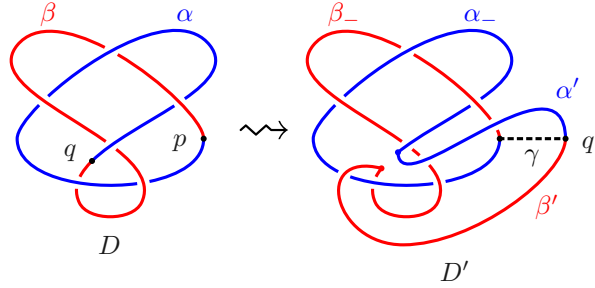
    \caption{Producing a standard meander diagram~$D'$ from a meander diagram~$D$ of the knot~$6_2$ in Step~2.}
    \label{fig:standard_meander}
\end{figure}

\vspace{5pt}
\underline{Step 3: Potholder diagrams.} To convert a standard meander diagram into a potholder diagram, consider the embedded circle $\alpha \cup \gamma \subset \mathbb S^2 = \RR^2 \cup \{\infty\}$, where~$\gamma$ is as in the definition of a standard meander diagram. We begin by performing an isotopy of~$\mathbb S^2$ that moves~$\alpha \cup \gamma$ to $(\{0 \} \times \RR) \cup \{\infty\}$, with the arc~$\alpha$ becoming a vertical interval, and $\gamma$~an unbounded region of the $y$
-axis. Interchanging the roles of $\alpha$~and~$\beta$ if necessary, we may assume that $\beta$~emanates from the lower change point towards the right of the $y$-axis, and with a type-1 Reidemeister move near the upper change point we may also arrange it so that $\beta$~emanates from it towards the left. Note that this forces the number~$n$ of crossings to be odd.
Since $\beta$~does not intersect the interior of~$\gamma$, it decomposes as a sequence of $n+1$~arcs with endpoints in~$\alpha$, that alternate between lying to its right and its left. See \Cref{fig:vertical_alpha} for an example.

\begin{figure}[h!]
    \centering
    \def\svgwidth{0.6\linewidth}
    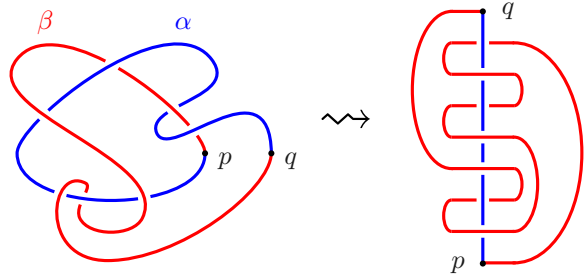
    \caption{Isotoping (in~$\mathbb S^2)$ a simplified version of the standard meander diagram from Figure~\ref{fig:standard_meander} to turn~$\alpha$ into a vertical line segment in Step~3.}
    \label{fig:vertical_alpha}
\end{figure}

\begin{figure}[h!]
    \centering
    \def\svgwidth{0.9\linewidth}
    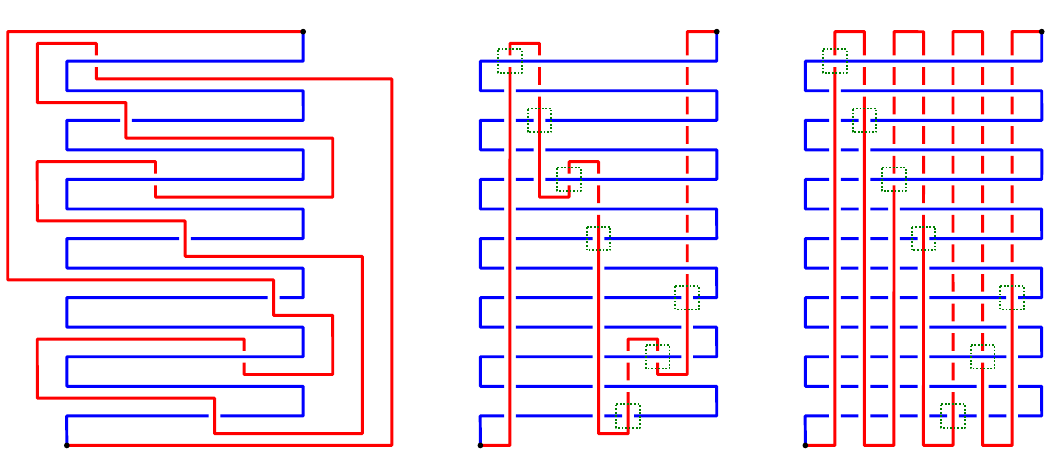
    \caption{The final stages of Step~3 of the construction of a potholder diagram, following the example in \Cref{fig:vertical_alpha}. Left: The vertical arc $\alpha$ is isotoped to form the horizontal zigzag of the target potholder diagram. Middle: The right (respectively left) sub-arcs of~$\beta$ are moved in front (respectively behind) the diagram, forming overcrossing $\cup$-shapes (respectively undercrossing $\cap$-shapes). For readability, pre-existing crossings are highlighted. Right: The cups and caps are extended to complete the potholder diagram of width~$7$ and height~$13$.}
    \label{fig:potholder_6_2}
\end{figure}

We aim to produce a potholder diagram of width~$n$ and height~$2n-1$. The first step is to isotope the picture to make~$\alpha$ the zigzagging arc containing the $2n-1$~horizontal strands of our target potholder diagram, and we do this in such a way that there is precisely one crossing at each of the~$n$ horizontal strands lying at odd height. These crossings are then adjusted horizontally, so that scanning the diagram from left to right, they appear in the order given by traversing~$\beta$. In our running example, the result looks as in \Cref{fig:potholder_6_2} (left). Observe that if one orients~$\beta$ from the bottom-left to the top-right of the diagram, then the $\beta$-strands at these crossings will alternate between pointing up and down.

Next, consider the sub-arcs of~$\beta$ lying to the right of~$\alpha$, and visualize a three-dimensional isotopy that moves them in front of the diagram, so that each projects into the vertical strip between the crossings it connects. The sub-arcs to the left of~$\alpha$ are isotoped behind the diagram in similar fashion. Since the crossings between $\alpha$ and~$\beta$ were arranged to be horizontally consecutive along~$\beta$, the projections of the newly isotoped arcs form no new crossings among themselves. Moreover, given how the orientations of the $\beta$-strands at these crossings alternate, the newly produced over-arcs will always form $\cup$-shapes and the under-arcs will form $\cap$-shapes. This results in a diagram as in \Cref{fig:potholder_6_2} (middle). Elongating these cups and caps yields the desired potholder diagram, as in \Cref{fig:potholder_6_2} (right).

It should be apparent, especially by staring at \Cref{fig:potholder_6_2} (middle), that the potholder diagram thus produced might very well be much larger than necessary. The reader may verify that the potholder diagram from our running example may be simplified to the one in \Cref{fig:6_2_simplified}.

\begin{figure}[h]
    \centering
    \def\svgwidth{0.15\linewidth}
    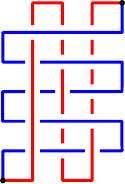
    \caption{A potholder diagram of width~$3$ and height~$5$ for the knot~$6_2$, obtained by simplifying the one in \Cref{fig:potholder_6_2}.}
    \label{fig:6_2_simplified}
\end{figure}


\subsection{All knots are trivial}

The last ingredient of our program is to convert pot\-holder diagrams into knot\-holder diagrams. The following proposition tells us we can do this even with some control on the sign tuple of the final diagram. 

\begin{prop}[Potholder to knotholder] \label{prop:potholder_to_knotholder}
    Let $m\in \NN_{\ge 1}$, let~$K$ be a knot that has a potholder diagram of width~$2m-1$, and let $\varepsilon = (\varepsilon_1,\ldots, \varepsilon_m)$ be an $m$-tuple of signs. Then $K$~has a knotholder diagram with $m$~$\cup$-shapes and sign tuple~$\varepsilon$.
\end{prop}
\begin{proof}

    We will illustrate the construction of the desired knot\-holder diagram from a potholder diagram of width~$5$, so $m=3$, and with a triple $\varepsilon = (-,-,+)$. It will then be clear how the method proceeds in general. The text should be followed alongside \Cref{fig:main_proof}.

\begin{figure}[h]
    \centering
    \def\svgwidth{0.9\linewidth}
    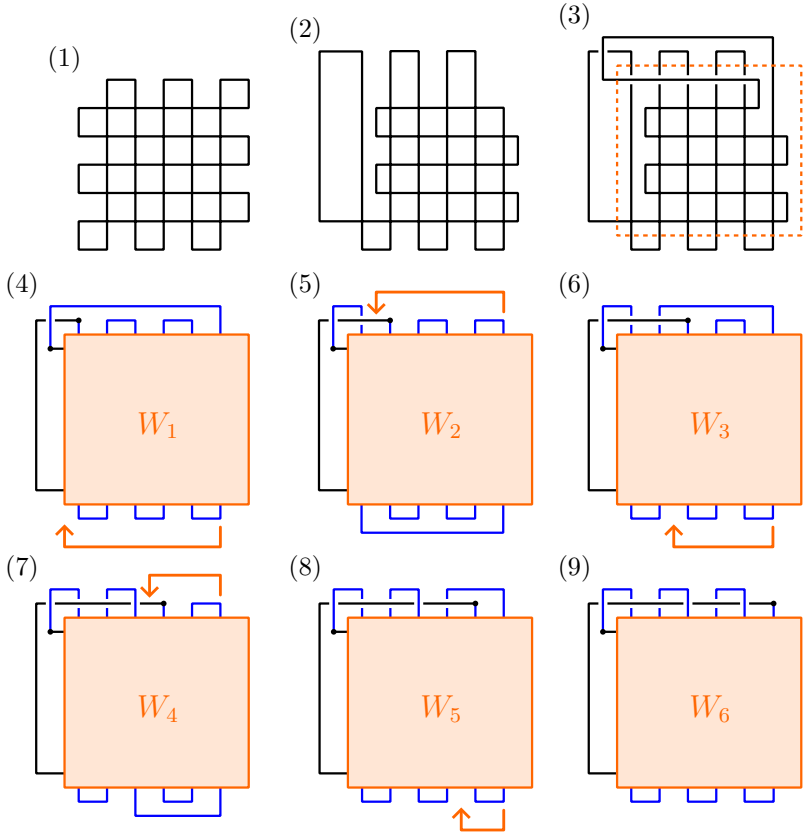
    \caption{Producing a knotholder diagram from a potholder diagram.}
    \label{fig:main_proof}
\end{figure}

    The first modification to the given potholder diagram (\Cref{fig:main_proof}~(1)) is to remove the superfluous top-right crossing. We then use a type-1 Reidemeister move on the bottom-left to produce a twist, and enlarge the resulting loop so it spans the height of the diagram, lying to its left. After a few additional minor adjustments to make room for later steps, the resulting diagram looks as in \Cref{fig:main_proof}~(2). The new crossing at the bottom-left may be taken to be either positive or negative, and the remaining ones are of course as in the original diagram.

    Next, we enlarge the region on the top right  to produce a loop extending to the left of the diagram as in \Cref{fig:main_proof}~(3). This loop is overcrossing at every intersection, reflecting the assumption that the last sign in our tuple is $\varepsilon_m = +$. In fact, the left-most crossing produced in this step will be the crossing numbered~$m$ in the knotholder diagram under construction. With this choice, we have ensured that its sign is positive. If we had $\varepsilon_m = -$, then we would have instead made this loop undercrossing at every intersection.

    \Cref{fig:main_proof}~(3) also highlights a weaver diagram~$W_1$ of width~$2m$. This region shall henceforth be blackboxed, but with each step we will ensure that it remains a width-$2m$ weaver diagram. 
    Adding some suggestive decorations to further hint at what roles each part of the diagram will play in the end, we obtain the schematic in \Cref{fig:main_proof}~(4).

    Our next goal is to produce the intersection numbered~$2m-1$ of the knot\-holder diagram, which should again have sign~$\varepsilon_m$. We do this by moving the right-most vertical strand to the position immediately to the right of the last produced intersection, as indicated by the arrow in \Cref{fig:main_proof}~(4). Again because $\varepsilon_m = +$, we move this strand across the back of the diagram, producing a positive crossing.
    Throughout this motion, other regions of the knot get dragged along; namely, we need to ponder on how $W_1$~changes. For the corresponding weaver tangle, which we also denote by~$W_1 \subset C$, this corresponds to dragging the right-most segment~$L_{2m}$ of the warp to position~$1$ along the back of the cylinder~$C$, moving all other~$L_i$ one position to the right, and dragging the weft along. More precisely, we modify~$W_1$ by acting with the element of $\Mcg(D, \{x_1, \ldots, x_{2m}\})$ represented by the braid~$b_1\in \mathcal B_{2m}$ shown in \Cref{fig:braids}. The weaver tangle $W_2 \coloneq b_1 \cdot W_1$ thus produced appears in the resulting diagram, which we show in \Cref{fig:main_proof}~(5).

    \begin{figure}[h!]
    \centering
    \def\svgwidth{0.8\linewidth}
\begingroup%
  \makeatletter%
  \providecommand\color[2][]{%
    \errmessage{(Inkscape) Color is used for the text in Inkscape, but the package 'color.sty' is not loaded}%
    \renewcommand\color[2][]{}%
  }%
  \providecommand\transparent[1]{%
    \errmessage{(Inkscape) Transparency is used (non-zero) for the text in Inkscape, but the package 'transparent.sty' is not loaded}%
    \renewcommand\transparent[1]{}%
  }%
  \providecommand\rotatebox[2]{#2}%
  \newcommand*\fsize{\dimexpr\f@size pt\relax}%
  \newcommand*\lineheight[1]{\fontsize{\fsize}{#1\fsize}\selectfont}%
  \ifx\svgwidth\undefined%
    \setlength{\unitlength}{525.82560382bp}%
    \ifx\svgscale\undefined%
      \relax%
    \else%
      \setlength{\unitlength}{\unitlength * \real{\svgscale}}%
    \fi%
  \else%
    \setlength{\unitlength}{\svgwidth}%
  \fi%
  \global\let\svgwidth\undefined%
  \global\let\svgscale\undefined%
  \makeatother%
  \begin{picture}(1,0.15001462)%
    \lineheight{1}%
    \setlength\tabcolsep{0pt}%
    \put(0,0){\includegraphics[width=\unitlength,page=1]{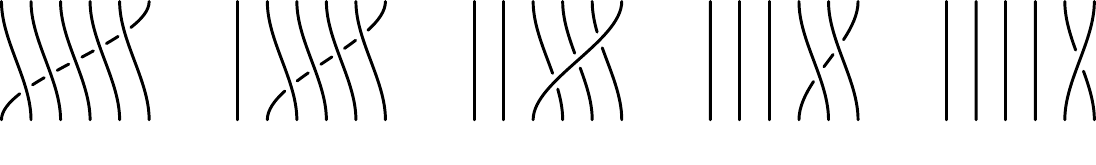}}%
    \put(0.07142826,0.00245149){\color[rgb]{0,0,0}\makebox(0,0)[t]{\lineheight{0.69999999}\smash{\begin{tabular}[t]{c}$b_1$\end{tabular}}}}%
    \put(0.28706216,0.00245149){\color[rgb]{0,0,0}\makebox(0,0)[t]{\lineheight{0.69999999}\smash{\begin{tabular}[t]{c}$b_2$\end{tabular}}}}%
    \put(0.50269606,0.00245149){\color[rgb]{0,0,0}\makebox(0,0)[t]{\lineheight{0.69999999}\smash{\begin{tabular}[t]{c}$b_3$\end{tabular}}}}%
    \put(0.71832996,0.00245149){\color[rgb]{0,0,0}\makebox(0,0)[t]{\lineheight{0.69999999}\smash{\begin{tabular}[t]{c}$b_4$\end{tabular}}}}%
    \put(0.93396387,0.00245149){\color[rgb]{0,0,0}\makebox(0,0)[t]{\lineheight{0.69999999}\smash{\begin{tabular}[t]{c}$b_5$\end{tabular}}}}%
  \end{picture}%
\endgroup%

    \caption{The braids used in successively modifying the weaver tangles $W_1, \ldots, W_{2m}$. For each $i$, we have $W_{i+1} \coloneq  b_i \cdot W_i$.}
    \label{fig:braids}
\end{figure}
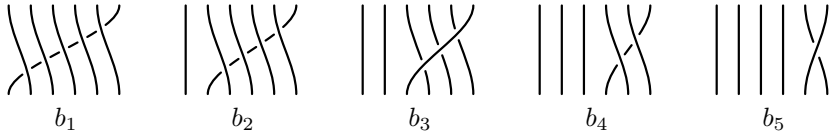

    Next, we consult the second-to-last sign~$\varepsilon_{m-1}$, which in our running example instructs us to produce two negative intersections. Each of them is obtained, as before, by moving the right-most vertical strand to the location immediately to the right of the previously produced intersection. Since we want these intersections to be negative, we move the first of these strands across the back of the diagram, and the second across the front. This yields the diagrams in \Cref{fig:main_proof} (6)~and~(7), respectively, where $W_3\coloneq b_2\cdot W_2$ and $W_4 \coloneq b_3 \cdot W_3$ for the braids~$b_i$ of \Cref{fig:braids}.

    We continue in this fashion, reading the signs of~$\varepsilon$ from left to right, and each time producing a pair of intersections and modifying the weaver diagram by the appropriate braids. In our example, this amounts to dragging one more pair of strands, producing the diagrams in \Cref{fig:main_proof} (5)~and~(6), where as before $W_{i+1} \coloneq b_i\cdot W_i$. After running through the whole tuple of signs, we reach a knotholder diagram for~$K$.
\end{proof}

In practice, one does not care so much about the sign tuple in the resulting knotholder diagram; it is rather more interesting to seek ``simple'' diagrams that render the trick easier to memorize and execute. 
This is taken into account in \Cref{fig:proof_example}, where we exemplify the algorithm on a potholder diagram for our running example~$6_2$. We chose not to start from the diagram in \Cref{fig:6_2_simplified} because that would lead to too much simplification, defeating the illustrative purpose of the example. A more convenient knotholder diagram for this knot is shown in \Cref{fig:intro_list_of_small_knots}.
Combining \Cref{thm:universality_potholder} and \Cref{prop:potholder_to_knotholder} yields:

\knotholderuniversality*

\begin{rem}[Band surgery on unknots]
    In light of \Cref{rem:band_surgery}, a consequence of \Cref{thm:main_intro} is that for every knot~$K$, there is a band surgery of an unknot that produces a $2$-component link of~$K$ and another unknot. It is not difficult to see that this is equivalent to the statement that every knot~$K$ may be obtained from some band surgery on a link of two unknots. 
    In a forthcoming article \cite{AQ}, we give a vast generalization of this fact.
\end{rem}

\section*{Outlook: Are all links trivial?}\label{sec:outlook}

Following a question posed by an anonymous referee, we wonder if a troupe of \(n\)~magicians could extend the ideas of this paper to produce any \(n\)-component link from an unlink of \(n\)~unknots. For example, one could conceive of a diagram as in \Cref{fig:whitehead_link} as an instruction for how two magicians can produce a Whitehead link by a sleight of hand. We leave the exploration of a formal treatment to the intrepid knotholders.

\begin{figure}[h!]
    \centering
    \def\svgwidth{0.3\linewidth}
    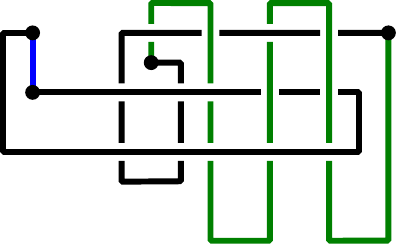
    \caption{A link diagram for the Whitehead link built out of two knotholder diagrams for the unknot. Only the performer on the left, holding the blue string, needs to execute a sleight of hand.}
    \label{fig:whitehead_link}
\end{figure}

    \begin{figure}[h!]
    \centering
    \def\svgwidth{0.9\linewidth}
    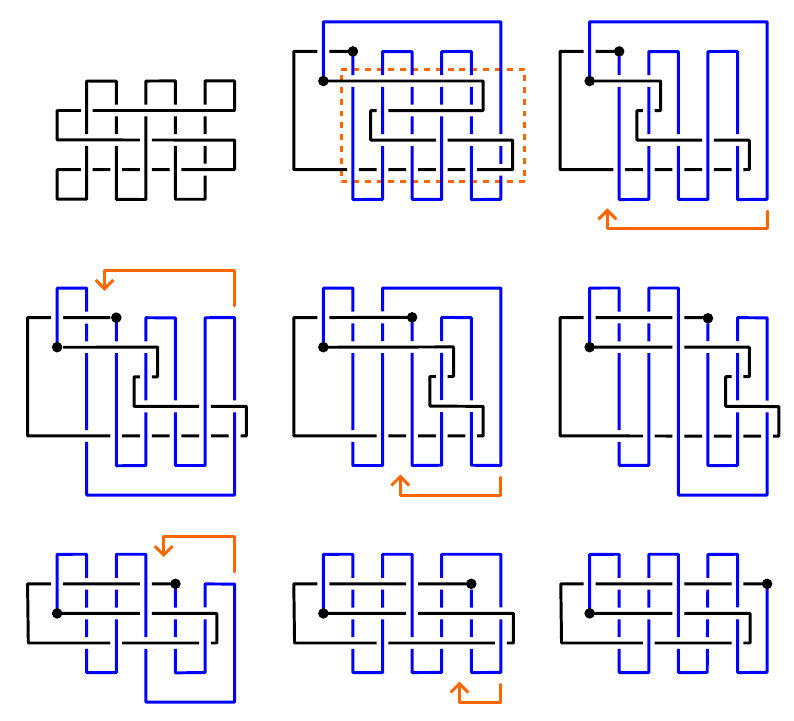
    \caption{Producing a knotholder diagram for~$6_2$ from a potholder diagram. We start with the potholder diagram obtained from the one in \Cref{fig:6_2_simplified} by reflecting along a diagonal line and flipping all crossings (1). In the preparatory steps (2), the large loop in the upper region was chosen to lie in front of the diagram, and so the last entry in the sign tuple of the knotholder diagram will be positive. With this choice, the relevant weaver diagram, which has been highlighted, allows for some simplification (3). Similarly, the diagrams in (5) and~(8) are produced from their predecessors by making an arbitrary choice (but not the ones in (4), (6) and (9)!), leading to a knotholder diagram with sign tuple~$(-,-,+)$. Note that even though the diagram in~(7) represents the same knot as the one in~(6), it is obtained by replacing the weaver diagram with a simpler one that does not represent an isotopic weaver tangle.}
    \label{fig:proof_example}
\end{figure}

\newpage 
\phantom{.}
\newpage

\printbibliography

\flushleft
---------

\textsc{Raphael Appenzeller},  \texttt{\href{mailto:rappenzeller@mathi.uni-heidelberg.de}{rappenzeller@mathi.uni-heidelberg.de}}

\textsc{José Pedro Quintanilha},   \texttt{\href{mailto:jquintanilha@mathi.uni-heidelberg.de}{jquintanilha@mathi.uni-heidelberg.de}}

\vspace{8pt}
Institut für Mathematik,
Im Neuenheimer Feld 205,

69120 Heidelberg, Germany
\end{document}